\documentclass[a4paper,pdftex,reqno,11pt]{amsart}

\usepackage[english]{babel}
\usepackage[T1]{fontenc}
\usepackage{amsmath,amsfonts,amssymb,amsthm}
\usepackage{cite}
\usepackage{supertabular}
\usepackage{url}
\usepackage{geometry}

\newtheorem{theorem}{Theorem}
\newtheorem{lemma}[theorem]{Lemma}
\newtheorem{observation}[theorem]{Observation}

\newcommand{\gauss}[3]{\genfrac{[}{]}{0pt}{}{#1}{#2}_{#3}}
\newcommand{\GF}{\mathbb{F}}
\newcommand{\PG}[2]{\operatorname{PG}(#1,#2)}

\begin{document}
\title[A subspace code of size $333$ in the setting of a binary Fano plane]{A subspace code of size $\mathbf{333}$ in the setting of a binary $\mathbf{q}$-analog of the Fano plane}
\date{}
\author{Daniel~Heinlein, Michael~Kiermaier, Sascha~Kurz, and Alfred~Wassermann$^\star$}
\thanks{$^\star$ All authors are with the Department of Mathematics, Physics, and Computer Science, University of Bayreuth, Bayreuth, GERMANY. 
  Email: firstname.lastname@uni-bayreuth.de
  The work was supported by the ICT COST Action IC1104 and grants KU 2430/3-1, WA 1666/9-1 -- {\lq\lq}Integer Linear Programming Models for 
  Subspace Codes 
  and Finite Geometry{\rq\rq} -- from the German Research Foundation.}

\begin{abstract}
We show that there is a binary subspace code of constant dimension 3 in ambient dimension 7, having minimum subspace distance 4 and cardinality 
333, i.e., $333 \le A_2(7,4;3)$, which improves the previous best known lower bound of 329. Moreover, if a code with these parameters has at 
least 333 elements, its automorphism group is in one of 31 conjugacy classes.

This is achieved by a more general technique for an exhaustive search in a finite group that does not depend on the enumeration of all subgroups.

\medskip
  
\noindent
  \textbf{Keywords:} Finite groups, finite projective spaces, constant dimension codes, subspace codes, subspace distance, combinatorics, computer search.\\
  \textbf{MSC:} 51E20; 05B07, 11T71, 94B25
\end{abstract}

\maketitle

\section{Introduction}
Since the seminal paper of K\"otter and Kschischang \cite{koetter2008coding} there is a still 
growing interest in subspace codes, which are sets of subspaces of the $\GF_q$-vector space $\GF_q^n$ 
together with a suitable metric. 
If all subspaces, which play the role of the codewords, have the same dimension, say $k$, then one speaks of 
constant dimension codes. The, arguably, most commonly used distance measures for subspace codes, motivated 
by an information-theoretic analysis of the Koetter-Kschischang-Silva model, see e.g.\ \cite{silva2008rank}, 
are the \emph{subspace distance} 
$$
  d_S(U,W):=\dim(U+W)-\dim(U\cap W)=2\cdot\dim(U+W)-\dim(U)-\dim(W)
$$ 
and the
\emph{injection distance} 
$$
  d_I(U,W):=\max\left\{\dim(U),\dim(W)\right\}-\dim(U\cap W),
$$ 
where $U$ and $W$ are subspaces of $\GF_q^n$. For constant dimension codes we have $d_S(U,W)=2d_I(U,W)$, so 
that the subsequent results are valid for both distance measures. By $A_q(n,d;k)$ we denote the maximum cardinality 
of a constant dimension code in $\GF_q^n$ with subspaces of dimension $k$ and minimum subspace distance $d$. From a 
mathematical point of view, one of the main problems of subspace coding is the determination of the exact value of 
$A_q(n,d;k)$ or the derivation of suitable bounds, at the very least. 

Currently, there are just a very few, but 
nevertheless very powerful, general construction methods available, see e.g.~\cite{MR2589964,heinlein2017asymptotic} 
for the details of the Echelon-Ferrers and the linkage construction. Besides that, 
several of the best known constant dimension codes for moderate parameters have been found by prescribing a subgroup 
of the automorphism group of the code, see e.g.~\cite{MR2796712}. However, the prescribed subgroups 
have to be chosen rather skillfully, since there are many possible choices and some groups turn out to permit only 
small codes. 

Here, we aim to develop a systematic approach, i.e., we want to check \emph{all} groups, exceeding some 
problem-dependent cardinality. For some fixed parameters $q$, $n$, $k$, and $d$ this is a finite problem -- in theory. 
As the problem for the exact determination of $A_q(n,d;k)$ is finite too, one quickly reaches computational limits. 
Even the generation of all possible groups becomes computationally intractable for rather moderate parameters. In this 
paper we describe a toolbox of theoretical and computational methods how to determine the best constant dimension codes 
admitting an arbitrary automorphism group of reasonable size, partially overcoming the inherent combinatorial explosion 
of the problem. 

Most of the techniques will be rather general. However, for our numerical computations we will focus on 
the specific set of parameters of $A_2(7,4;3)$, which is the smallest undecided case for binary constant dimension 
codes.\footnote{The parameters $n$, $k$, and $d$ have to satisfy $1\le k\le n$, $d\equiv 0\pmod 2$, and $2\le d\le 2k$. 
Taking all $\gauss{n}{k}{q}$ $k$-dimensional subspaces of $\GF_q^n$ yields $A_q(n,2;k)=\gauss{n}{k}{q}$. The case 
$d=2k$ corresponds to partial $k$-spreads, i.e., trivially intersecting unions of $k$-dimensional subspaces of $\GF_q^n$. 
For $q=2$ the maximum possible cardinalities are known for $n<11$ and the smallest undecided case is $129\le A_2(11,8;4)\le 132$, 
see e.g.~\cite{beutelspacher1975partial,kurzspreads,kurzspreadsII}. The first non-trivial and non-spread case 
$A_2(6,4;3)=77$ was treated in \cite{hkk77}. The corresponding five isomorphism types of optimal codes have been 
classified by a mixture of theoretical arguments and severe computer computations.} Prior to this paper, the best 
known bounds were $329\le A_2(7,4;3)\le 381$.\footnote{See \url{http://subspacecodes.uni-bayreuth.de} and the corresponding 
technical manual~\cite{TableSubspacecodes} for an on-line table of known bounds on $A_q(n,d;k)$.}  During our systematic 
approach we found a corresponding code of cardinality $333$. In the language of projective geometry, see e.g.~\cite{etzionsurvey,greferath2018network} 
for recent surveys,  those codes correspond to collections of planes in $\PG{6}{2}$ mutually intersecting 
in at most a point. $381$ such planes would correspond to a binary $q$-analog of the Fano plane, whose existence is still 
unknown. In dimension $n=13$ a binary $q$-analog of a Steiner system was shown to exist in \cite{MR3542513}. For our parameters 
in dimension $n=7$ it was shown recently in \cite{kiermaier2016order} that a (still) possible binary $q$-analog of the 
Fano plane has an automorphism group of order at most $2$.

With respect to the concrete parameters, the main contributions of our paper are:

\begin{theorem}\label{main_thm_1}
Let $C$ be a set of planes in $\operatorname{PG}(6,2)$ mutually intersecting in at most a point.
If $\left|C\right|\geq 329$, then the automorphism group of $C$ is conjugate to one of the $33$ subgroups of $\operatorname{GL}(7,2)$ given in Appendix~\ref{sec_app_survivinggroups}.
The orders of these groups are $1^1 2^1 3^2 4^7 5^1 6^3 7^2 8^{11} 9^2 12^1 14^1 16^1$ denoting the number of cases as exponent.
Moreover, if $\left|C\right| \geq 330$ then $\left|\operatorname{Aut}(C)\right| \leq 14$ and if $\left|C\right| \geq 334$ then $\left|\operatorname{Aut}(C)\right| \leq 12$.
\end{theorem}

\begin{theorem}\label{main_thm_2}
In $\operatorname{PG}(6,2)$, there exists a set $C$ of $333$ planes mutually intersecting in at most a point.
Hence,
\[A_2(7,4;3) \ge 333\text{.}\]
The set $C$ is given explicitly in Appendix~\ref{sec_app_333}.
Its automorphism group $\operatorname{Aut}(C)$ is isomorphic to the Klein four-group.
It is the group $G_{4,6}$ in Appendix~\ref{sec_app_survivinggroups}.
\end{theorem}

The remaining part of the paper is structured as follows.
In Section~\ref{sec_concrete_parameters} we review the previous work done on binary constant dimension codes for our parameters $n=7$, $d=4$, and $k=3$.
Preliminaries and utilized methods are described in Section~\ref{sec_preliminaries}.
In Section~\ref{sec_ilp}, a method is described how to determine whether a code with a prescribed automorphism group and size exists.
In our analysis of the possible groups (eventually) admitting a code of size at least $329$, we start with groups of prime power order in 
Section~\ref{sec_prime_power_order} and continue with groups of non-prime-power order in Section~\ref{sec_composite_order}.
The modifications described in Section~\ref{sec_modifying_codes} of a code of size $329$ yield the code mentioned in 
Theorem~\ref{main_thm_2} and Appendix~\ref{sec_app_333}.
We draw conclusions and mention some open problems for further research in Section~\ref{sec_conclusion}.
The groups corresponding to Theorem~\ref{main_thm_1}, as well as the code of size $333$ of Theorem~\ref{main_thm_2}, are listed in the appendix.

\section{Previous work}
\label{sec_concrete_parameters}
The upper bound $A_2(7,4;3)\le 2667/7 = 381$ can be concluded by observing that there are
$2667$ 2-dimensional subspaces in $\mathbb{F}_2^7$ and every codeword contains seven 2-dimensional subspaces.

Equality is attained if each $2$-dimensional subspace is covered by exactly one codeword.
This would be a binary $q$-analog Steiner triple system $S_2(2,3,7)$.
In the \emph{limiting case} {\lq}$q=1${\rq} such a structure is well known and corresponds to subsets of $\{1,\dots,n\}$. 
It is the famous Fano plane.
The only known $q$-analogs of Steiner systems have parameters $S_2(2,3,13)$~\cite{MR3542513}.
The existence question for a $2$-analog Steiner triple system $S_2(2,3,7)$ has been tackled 
in several research papers, see 
e.g.~\cite{etzion2015new,etzion2015structure,MR2801584,MR3468803,MR3444245,MR3403762,MR1725002,MR1305448,MR908977,MR1419675}.
In~\cite{MR3398870,kiermaier2016order} the authors eliminated all but one non-trivial group 
as possible automorphism groups of a binary $q$-analog of the Fano plane, 
so that the automorphism group is known to be at most of order two.

Relaxing the condition ``equal'' to ``at most'', we arrive at
binary constant dimension codes with parameters $n=7$, $d=4$, and $k=3$.
The construction of~\cite{MR2589964} gives $A_2(7,4;3)\ge 289$.
In 2008 Etzion and Vardy~\cite{MR2796712} found a code of cardinality $294$.
A code of cardinality $304$ was found in~\cite{MR2796712} via the prescription of a cyclic group of order~$21$.
Prescribing a cyclic group of order~$15$ and modifying corresponding codes yields 
$A_2(7,4;3)\ge 329$~\cite{MR3198748}.
In the sequel, an explicit, computer-free construction of (a different) code of size $329$ 
was presented in~\cite{liu2014new,MR3444245}. For more details on the underlying 
expurgation-augmentation method see \cite{ai2016expurgation}.  
Hitherto, all known examples of codes of cardinality $329$ only admit the trivial automorphism.

In the following, we use a similar approach and reformulate the corresponding problem 
as an integer linear programming problem, see Section~\ref{sec_modifying_codes},
and succeed 
to construct a code of cardinality $333$ starting from a code of size $329$.

\section{Preliminaries}\label{sec_preliminaries}
Let $V=\mathbb{F}_q^n$ be the standard vector space of dimension $n \ge 3$. Let $C$ be 
a set of subspaces of $V$ and $K$ be a subspace of $V$. The fundamental theorem of 
projective geometry~\cite{MR0082463,MR0052795} states that the 
set of order preserving isometries is $\operatorname{P\Gamma L}(V)$. Let $q=2$ throughout this paper. Then we have 
$\operatorname{P\Gamma L}(\mathbb{F}_2^n) = \operatorname{GL}(\mathbb{F}_2^n)$ and, 
after choosing a basis of $V$, the elements in this group can be represented as matrices.
By \[U^g = g^{-1} U g \quad\text{ and }\quad U^G=\{U^g \mid g \in G\}\] we denote the \emph{conjugation} of 
$U \le \operatorname{P\Gamma L}(V)$ with $g \in \operatorname{P\Gamma L}(V)$ and $G \le \operatorname{P\Gamma L}(V)$.

For the bijective map $r$ that maps $\gauss{V}{k}{}$ to binary $k \times n$ matrices in reduced row echelon form with rank $k$ and the operation $\operatorname{RREF}$ that maps a matrix to its reduced row echelon form, the operation of $M \in \operatorname{GL}(V)$ on $K \in \gauss{V}{k}{}$ is given by matrix multiplication $r^{-1}(\operatorname{RREF}(r(K) \cdot M))$.

An element $M \in \operatorname{P\Gamma L}(V)$ is called \emph{automorphism} of $C$ if $M$ stabilizes $C$, i.e., $C\cdot M=C$.
A subgroup $U \le \operatorname{P\Gamma L}(V)$ is called \emph{an automorphism group} of $C$ if each $M \in U$ 
is an automorphism of $C$ and it is called \emph{the automorphism group} of $C$, 
$\operatorname{Aut}(C)$, if it contains all automorphisms of $C$.

For a subgroup $U \le \operatorname{P\Gamma L}(V)$, \[K \cdot U=\{K\cdot M \mid M \in U\} \quad\text{ and }\quad C \cdot U=\{K\cdot U \mid K \in C\}\] denote the \emph{orbits} of $K$ and $C$.
The \emph{orbit space} of all $k$-dimensional subspaces of $V$ and $U \le \operatorname{P\Gamma L}(V)$ is denoted as $\gauss{V}{k}{} / U$.

By $A_q(n,d;k;U)$ we denote the maximum size of a constant dimension code $C$ in $\gauss{V}{k}{}$ with subspace distance at least $d$ and $U \le \operatorname{Aut}(C)$.
Note that $A_q(n,d;k;I) = A_q(n,d;k)$ where $I$ is the identity subgroup in $\operatorname{P\Gamma L}(V)$.

This paper uses two obvious but far reaching observations.
\begin{observation}\label{obs_monotonicity}\leavevmode
\begin{enumerate}
\item $A_q(n,d;k;M) \ge A_q(n,d;k;N)$ for $M \le N \le \operatorname{P\Gamma L}(V)$ and
\item $A_q(n,d;k;U^g)=A_q(n,d;k;U)$ for all $g \in \operatorname{P\Gamma L}(V)$.
\end{enumerate}
\end{observation}

For example the $32,252,031$ groups (or elements) of order two in $\operatorname{P\Gamma L}(\mathbb{F}_2^7) = \operatorname{GL}(\mathbb{F}_2^7)$ fall in just three conjugacy classes.


Occasionally, we will mention \emph{abstract types} of groups. 
We use $Z_n$ for the cyclic group, $D_n$ for the dihedral group, 
$Q_n$ for the quaternion group of order~$n$, $A_n$ for the alternating group, and 
$S_n$ for the symmetric group on $n$ elements. 
$\times$ denotes a direct product and $\rtimes$ denotes a (not necessarily unique) semidirect product of groups.

Given the \emph{abstract type} of a group, we can obtain precise information on the 
abstract types of its subgroups from the \emph{Small Groups library}~\cite{smallgroupslibrary_homepage}, implemented 
in the computer algebra system \texttt{Magma}, containing all groups with order at most~$2000$ except $1024$.

For an orbit space $X \cdot G$ the \emph{orbit type} is a number $c_1^{n_1}\cdot \ldots \cdot c_m^{n_m}$ with the meaning that $X \cdot G$ contains exactly $n_i$ orbits of cardinality $c_i$ for $i \in \{1,\ldots,m\}$ and no other orbits.

Using the observations above one can exclude all supergroups and their conjugates of a group $U$ as automorphism group 
of a subspace code of size at least $329$, as soon as $U$ can be excluded as possible automorphism group 
of such a code with the Kramer-Mesner like computation method of Section~\ref{sec_ilp}.
With this, the general idea is to (implicitly) consider all possible groups of automorphisms.

In order to formalize our approach from a more general point of view, we introduce a conjugation-invariant and monotone mapping
$\mathcal{P}$. 
For a group $U \le G$ we set 
\begin{itemize}
\item $\mathcal{P}(U)=0$, if $A_2(7,4;3;U)\le \kappa$, where we use $\kappa=328$ in this paper, 
\item $\mathcal{P}(U)=1$, if there is a code with code size $>\kappa$ such that $U$ is contained in its automorphism group or 
the computation was aborted after, say, $\Lambda$ hours. In this paper we use $\Lambda=48$.
\end{itemize}
Our strategy now is to systematically determine $P(U)$ for all subgroups $U\le G$ from the bottom up 
where we can stop the search, i.e. set $P(U)=0$, in the following cases:
\begin{enumerate}
\item If $U$ contains a subgroup whose order is in $S\subseteq\mathbb{N}$ and $\mathcal{P}(H)=0$ for all groups $H \le G$ of order $|H|\in S$. 
\item If $U$ contains a subgroup whose abstract type is in the set $T$ and $\mathcal{P}(H)=0$ for all groups $H \le G$ of type $t\in T$.
\item If $U$ contains a subgroup $H$ with $\mathcal{P}(H)=0$.
\end{enumerate}
Since only cardinalities of subgroups of $U$ need to be known in Step~(1), the theorems of Sylow and Hall,
see \cite[Section 4.2 and Thm. 9.3.1]{Hall1959}  are applied.
If the abstract type of $U$ is known, the Small Groups library can give the desired information for Step~(1).
If Step~(1) was not successful, then one can refine to the abstract type of $U$ in Step~(2).
Finally, the concrete conjugacy class of $U$ has to be known for Step~(3). Since Step~(3) is the computationally 
most expensive step, the more specialized and computationally cheap tests of Step~(1) and Step~(2) are introduced.

If $\mathcal{P}(U)$ is still undecided after all three steps, then the optimization problem from Section~\ref{sec_ilp} has to be solved.

From the group-theoretic point of view it remains to describe how the conjugacy classes of groups are generated.
For $p$-Sylow groups we need a single example since all of these groups are conjugate.
For cyclic subgroups we describe some shortcuts in Section~\ref{subsec_cyclic}.
Except for orders~$16$, $32$, and~$64$ the built-in functions of \texttt{Magma} are sufficient to produce 
the required list of conjugacy classes of groups for our parameters.
For the remaining powers of two we provide a general algorithmic tool in Subsection~\ref{subsec_generating_groups}.
Here, the idea is to extend a list of groups, having $\mathcal{P}(\cdot)=1$, to a complete list $L$ of larger groups 
of a desired order~$u$ such that all groups of order~$u$ which are not conjugate to elements of $L$ have $\mathcal{P}(\cdot)=0$.

We remark, that the definition of $\mathcal{P}(U)$ easily generalizes to the determination of $A_q(n,d;k;U)$. 
Observation~\ref{obs_monotonicity} gives the necessary monotonicity and conjugation invariance.

\subsection{Generating groups up to conjugacy}\label{subsec_generating_groups}

Let $f:\{A \le G\} \rightarrow \{0,1\}$ be a map such that
$f(A) \ge f(B)$ for all $A \le B$ and $f(A)=f(A^g)$ for all $g \in G$,

\begin{lemma}\label{key_lemma}
Let $G$ be a finite group.
Furthermore, let $t$, $u$ be integers with $t \mid u \mid |G|$ such that
any subgroup of $G$ of order~$u$ contains a normal subgroup of order~$t$.

Suppose that the set ${\mathcal T}$ consists of all conjugacy classes of subgroups $T\leq G$ of order~$t$ 
such that $f(T)=1$.
Let ${\mathcal T}_N$ be a transversal of the orbits 
under the action of $G$.
Let
\[
{\mathcal U} \!=\! \{ U^{N_G(T)} \!\mid\!
    (T, N_G(T)) \in {\mathcal T}_N, 
    T \leq U \leq N_G(T),
    |U|\!=\!u\}\text{.}
\]
Then, $f(U)=0$ for all $U \le G$ with $|U| = u$ and $U^G \not \in {\mathcal U}$.
%
%

\end{lemma}
\begin{proof}
Assume there is a $U\le G \setminus {\mathcal U}$ with cardinality $u$ and $f(U)=1$, 
then it contains a normal subgroup $T$ of cardinality $t$ and by monotony $f(T)=1$. 
It follows that $(T, N_G(T))$ represents a conjugacy class in ${\mathcal T}$.
Moreover, since $T$ is a normal subgroup in $U$ and $N_G(T)$ is the largest subgroup of $G$ having $T$ as a normal subgroup, $U\le N_G(T)$. 
Hence, $U^{N_G(T)} \in {\mathcal U}$, contradicting the assumption.
%
\end{proof}

\textbf{Remark}:
If $u/t$ is a prime, then ${\mathcal T}_N$ can be restricted to 
the conjugacy classes of
$N_G(T)$ operating on its \emph{cyclic} subgroups.

\medskip
The requirements of this lemma on $t$ and $u$ may be fulfilled in certain constellations 
with the help of the Sylow Theorems 
see e.g.~\cite[Section 4.2]{Hall1959} or the Theorem of Hall, see~\cite[Theorem 9.3.1]{Hall1959}.
If neither the Sylow theorems nor the Hall theorem can be applied, 
the \textsl{Small Groups library}~\cite{smallgroupslibrary_homepage} may be of help.
For example, it contains the information that 
any group of order~$20$ has a normal subgroup of order~$5$ or~$10$. 
Also, any group of order~$40$ has a normal subgroup of order~$2$, $5$, $10$, or~$20$.

We will use Lemma \ref{key_lemma} to handle the possible automorphism groups of order~$16$, $32$, and $64$.
\subsection{Techniques for an exhaustive search in a finite group}\label{subsec_cyclic}
Since we apply this technique to $G=\operatorname{GL}(\mathbb{F}_2^7)$, we profit from the 
special group structure of $\operatorname{GL}(\mathbb{F}_q^n)$. First, all
elements up to conjugacy can be generated by the normal forms, e.g., the Frobenius normal 
form~\cite{storjohann1998n}\footnote{The 
group $G_{4,6}$ from Appendix~B may also be written as $ \left\langle 
\left(
\begin{smallmatrix}
1 & 0 & 0 & 0 & 0 & 0 & 0 \\
0 & 0 & 1 & 0 & 0 & 0 & 0 \\
0 & 1 & 0 & 0 & 0 & 0 & 0 \\
0 & 0 & 0 & 0 & 1 & 0 & 0 \\
0 & 0 & 0 & 1 & 0 & 0 & 0 \\
0 & 0 & 0 & 0 & 0 & 0 & 1 \\
0 & 0 & 0 & 0 & 0 & 1 & 0
\end{smallmatrix}
\right),
\left(
\begin{smallmatrix}
1 & 0 & 0 & 0 & 0 & 0 & 0 \\
1 & 0 & 1 & 1 & 1 & 1 & 1 \\
1 & 1 & 0 & 1 & 1 & 1 & 1 \\
1 & 0 & 1 & 1 & 1 & 1 & 0 \\
1 & 1 & 0 & 1 & 1 & 0 & 1 \\
0 & 0 & 1 & 0 & 1 & 0 & 0 \\
0 & 1 & 0 & 1 & 0 & 0 & 0
\end{smallmatrix}\right)\right\rangle$, where the first generator is in Frobenius normal form.}.

Secondly, given an element $c\in G$, the check if a group $U \le G$ contains a conjugate of a cyclic 
subgroup $C=\langle c \rangle$ is easy.

We denote the eigenspace for the eigenvalue $1$, i.e., the fixed-point space, by $\operatorname{eig}(C,1)$. Note that 
$\dim(\operatorname{eig}(C,1))$ is invariant in the conjugacy class $C^G$. If for fixed integers $m$ and $n$ all cyclic subgroups 
$C \le G$ with $|C|=m$ and $\dim(\operatorname{eig}(C,1))=n$ are excluded, then all groups $U \le G$ having an element $c$ of order $m$ and 
$\dim(\operatorname{eig}(\langle c\rangle ,1))=n$ can be excluded as well. Furthermore this test replaces the expensive test 
for containment up to conjugacy.

In the remainder of this paper, we will simply speak of the dimension of the fixed-point space and use it in the context of 
cyclic groups and their conjugacy classes.

\section{An integer linear programming formulation for constant dimension codes with prescribed automorphisms}\label{sec_ilp}
In~\cite{MR2796712}, a computational method based on the Kramer-Mesner approach for large subspace codes with prescribed automorphism group is presented.
We adopt a similar method using an integer linear program (ILP) that provides lower and upper bounds on 
$A_2(7,4;3;U)$ for a prescribed automorphism subgroup $U\le G$.

Let $\gauss{\mathbb{F}_2^7}{3}{}$ and $\gauss{\mathbb{F}_2^7}{2}{}$ denote the set of all 
$3$-dimensional subspaces and $2$-dimensional subspaces in $\mathbb{F}_2^7$.
For a given group $U$ of prescribed automorphisms, let $T_3(U)$ be a transversal of the orbit space 
$\gauss{\mathbb{F}_2^7}{3}{} / U$ and $T_2(U)$ be a transversal of the orbit space $\gauss{\mathbb{F}_2^7}{2}{} / U$.
By $t(K,U)\in T_3(U)$ we denote the representative of the orbit containing $K\in \gauss{\mathbb{F}_2^7}{3}{}$.
As variables we choose $x_K\in\{0,1\}$, where $x_K=1$ if and only if the entire orbit $K\cdot U$ for $K\in T_3(U)$ is contained in the code.
The incidences are modeled with $M^U=(m_{T,K})_{T \in T_2(U), K \in T_3(U)}$ where 
\[
    m_{T,K} = |\{ W \in K \cdot U \mid T \le W \}|\text{.}
\]
Finding best constant dimension codes having this group of automorphisms can be formulated as an ILP, which easily 
generalizes to the determination of $A_q(n,d;k;U)$:
\begin{align*}
\text{ILP}(U)=&\max &\sum_{K \in T_3(U)} |K\cdot U|\cdot &x_K\\
&\text{s.t.} &M^U &x \le 1\\
&&&x_K \in \{0,1\} && \forall K\in T_3(U)
\end{align*}
By replacing the binary $x_K \in \{0,1\}$ by the weaker constraint $0\le x_K\le 1$ 
we obtain the so-called linear programming (LP) relaxation.

In case $m_{T,K} \ge 2$, the corresponding variable $x_K$ is trivially zero and consequently the orbit $K \cdot U$ is not in the code.

In order to compute $\mathcal{P}(U)$ for a given group, we first compute the optimal target value $z$ of the LP-relaxation, 
which can always be done in reasonable time. 
If $z<\kappa+1=329$ for the LP, then $\mathcal{P}(U)=0$. Otherwise we try to solve $\text{ILP}(U)$.
If an integral solution with target value at least $\kappa+1$ is found, or the computer search is abandoned after reaching a certain time limit,
then $\mathcal{P}(U)=1$. Otherwise we set $\mathcal{P}(U)=0$.

\subsection{Using the automorphisms of the orbit space}\label{subsec_symmetry_normalizer}
The prescription of a group $U \le \operatorname{GL}(\mathbb{F}_2^7)$ yields the orbit space $\gauss{\mathbb{F}_2^7}{3}{} / U$, 
which in turn has automorphisms.
It is well known that $N_{\operatorname{GL}(\mathbb{F}_2^7)}(U) \le \operatorname{Aut}(\gauss{\mathbb{F}_2^7}{3}{} / U)$.
These automorphisms can be used to reduce the overall solving time of the ILP. 

For this, let $O(U) := (\gauss{\mathbb{F}_2^7}{3}{} / U) / N_{\operatorname{GL}(\mathbb{F}_2^7)}(U)$ and
$t(o,U)$ be an arbitrary orbit of $O(U)$ containing $o\in \gauss{\mathbb{F}_2^7}{3}{} / U$.
For a $K$ in $t(o,U)$ the ILP from above is extended to $\text{ILP}_o$ by adding the constraint $x_{t(K,U)}=1$.

We will solve the $|O(U)|$ problems $\text{ILP}_o$. 
Thanks to the automorphisms this is sufficient to solve the initial ILP: 
$\mathcal{P}(U)=0 \Leftrightarrow \max\{z(\text{ILP}_o) \mid o \in O \} < \kappa +1$, 
where $z(\cdot{})$ denotes the objective value.
After choosing an ordering $\{o_1,\ldots,o_{|O(U)|}\}=O(U)$, 
processing $\text{ILP}_{o_i}$ yields additional information for the problems 
$\text{ILP}_{o_{i+1}}, \ldots, \text{ILP}_{o_{|O(U)|}}$.
If $z(\text{ILP}_{o_i}) \ge \kappa +1$ then we finish with $\mathcal{P}(U)=1$,
else no orbit in $o_i$ is part of any code with size at least $\kappa +1$ and can be excluded in the following $\text{ILP}_o$ by adding
the constraint
\begin{eqnarray}
x_{t(K,U)}=0 \text{ for a } K \in o' \text{ for all } o' \in o \text{.} \label{eqn_forb}
\end{eqnarray}
Therefore, the arrangement of these subproblems is important.
The goal is to have a small overall solving time, hence we sort $\{\text{ILP}_o \mid o \in O\}$ 
in decreasing size of $|o|$ and in case of equality decreasing in the number of forced codewords.
The first sorting criterion ensures few remaining automorphisms, due to the orbit-stabilizer theorem, 
whereas the second criterion ensures small computation times due to the fixtures.

To decrease the overall solving time even further, after determining the order of $\text{ILP}_o$, 
we assume that $\mathcal{P}(U)$ will be $0$ and generate all problems with the implied exclusions of~(\ref{eqn_forb}) 
beforehand and start solving them in parallel.
If there is an $o \in O$ with $z(\text{ILP}_o) \ge \kappa +1$, then our assumption was wrong and we return $\mathcal{P}(U)=1$.

\section{Groups of prime power order}\label{sec_prime_power_order}
We first start to consider groups of prime power order.
Due to $|\operatorname{GL}(\mathbb{F}_2^7)|=2^{21} \cdot 3^4 \cdot 5 \cdot 7^2 \cdot 31 \cdot 127$ it 
suffices to consider the primes $2$, $3$, $5$, $7$, $31$, and $127$. 
All necessary conjugates of subgroups were computed using \texttt{Magma}.

\subsection{Groups of order 5, 31, or 127}\label{subsec_r_1}

From the factorization of $|\operatorname{GL}(\mathbb{F}_2^7)|$ it follows that 
there is exactly one subgroup of $\operatorname{GL}(\mathbb{F}_2^7)$ up to conjugacy of order $5$, $31$, and $127$.

The group of order $127$ yields codes of maximum size $254$~\cite{MR2796712,MR908977}.

The group of order $31$ yields an orbit space of the $3$-dimensional subspaces of type $31^{381}$. 
The orbit space on the $2$-dimensional subspaces has the type $1^1 31^{86}$.
Solving the corresponding ILP yields a code of size $279$ which is also the maximum cardinality for this automorphism group.

The group of order $5$ has orbit type $1^1 5^{2362}$ on the $3$-dimensional subspaces and $1^7 5^{532}$ on the $2$-dimensional subspaces.
Unfortunately, this ILP is too difficult to solve in reasonable time.
Thus only $G_{5,1}$ (cf. Appendix~\ref{sec_app_survivinggroups}) remains.

\subsection{Groups of order $\mathbf{3^a}$ or $\mathbf{7^a}$}\label{subsec_p_3_7}
All groups of order $7$ are cyclic so that they can be computed using the Frobenius normal form. There are three non-conjugate groups. 
One of them can only yield codes of size at most 296 whereas the other two could not be excluded in reasonable time. A nontrivial 
element in the excluded group has a 4-dimensional fixed-point space and any non-trivial element of the non-excluded groups has 1-dimensional fixed-point spaces.

Since the maximum power of the prime 7 is 49 in $|\operatorname{GL}(\mathbb{F}_2^7)|$, there is exactly one subgroup of order 49 up to conjugacy. Using the Sylow theorems, it has to contain at least one subgroup of any conjugacy class of order 7. In particular it has to contain a conjugate to the previously excluded group of order 7. Therefore the group of order 49 cannot yield larger codes than 296.

The same can be performed for the groups of order 3. There are exactly three conjugacy classes of subgroups of order 3. One yields codes of cardinality at most 255. The other two groups could not be excluded in reasonable time.

There are exactly 4 groups of order 9 in the group $\operatorname{GL}(\mathbb{F}_2^7)$ up to conjugacy. Two of them contain the previously excluded group of order 3 and hence can only yield a largest code cardinality of 255. The other two groups of order 9 cannot be excluded. They have abstract type $Z_{9}$ and $Z_{3} \times Z_{3}$.

There are 3 conjugacy classes of groups of order 27.
One of them contains a conjugate of the excluded group of order 3.
With the methods of Section~\ref{sec_ilp}, we see that both groups yield codes of maximum size 309.

The unique conjugacy class of groups of order 81 contains a conjugate of the excluded group of order 3 and can therefore yield only codes of size at most 255.

Thus only $G_{7,1}$, $G_{7,2}$, $G_{3,1}$, $G_{3,2}$, $G_{9,1}$, and $G_{9,2}$ (cf. Appendix~\ref{sec_app_survivinggroups}) remain.

\subsection{Groups of order $\mathbf{2^a}$}\label{subsec_p_2_groups}

There are 3 conjugacy classes of groups of order 2 in $\operatorname{GL}(\mathbb{F}_2^7)$. 
The first cannot be excluded and has a 4-dimensional fixed-point space. 
The second can only yield codes of size 298 and has an 5-dimensional fixed-point space. 
The third can only yield codes of size 106 and has a 6-dimensional fixed-point space, cf.~\cite{MR2796712}.

There are $42$ conjugacy classes of subgroups of order $4$ in the group $\operatorname{GL}(\mathbb{F}_2^7)$. 
All but $8$ contain at least one already excluded group of order $2$, cf.\cite{MR3398870}. 
One of the remaining 8 groups can yield codes of size at most $327$.

There are $867$ conjugacy classes of subgroups of order~$8$ in the group $\operatorname{GL}(\mathbb{F}_2^7)$. 
All but $38$ contain an already excluded group of order~$2$. All but $11$ of the remaining groups
can be excluded computationally.

For the subgroups of order~$16$, we apply the technique described in the Section~\ref{subsec_generating_groups}.
Since a subgroup of index~$2$ is necessarily a normal subgroup, see e.g. \cite[Cor.~2.2.1]{Hall1959}, 
Lemma~\ref{key_lemma} can be applied for $t=8$ and $u=16$.
Up to conjugacy there are exactly $50$ subgroups of order $16$ of the group $\operatorname{GL}(\mathbb{F}_2^7)$
such that no contained $2$-subgroup is already excluded.
Solving the corresponding ILPs from Section~\ref{sec_ilp} shows that these $50$ subgroups 
can yield codes of cardinality at most $329$ and 
exactly one group attains this bound.

This group is of type $(Z_{4} \times Z_{2}) \rtimes Z_{2}$, see $G_{16,1}$ in the appendix, 
and it will play a major role in the process of finding the code of cardinality 333. 
In fact, there are up to isomorphism exactly $12$ codes of size $329$ under prescription of $G_{16,1}$. 
Each code has the orbit type $1^1 2^2 4^9 8^8 16^{14}$ and each of the $12$ isomorphism classes has 16 codes, 
summing up to a total of $192$ codes, which have $G_{16,1}$ as automorphism group.

Stepping the 2-Sylow ladder further up by applying Lemma~\ref{key_lemma} to $G_{16,1}$ 
with $t=16$ and $u=32$, we found a group of order $32$ that yields a code of size $327$ 
and by applying Lemma~\ref{key_lemma} to this group, we found a group of order $64$ that yielded a code of size $317$.

Thus only $G_{2,1}$, $G_{4,1}$, \ldots, $G_{4,7}$, $G_{8,1}$, \ldots, $G_{8,11}$, and $G_{16,1}$ (cf. Appendix~\ref{sec_app_survivinggroups}) remain.

\section{Groups of non-prime-power order}\label{sec_composite_order}
Using the Sylow theorems~\cite[Thm. 4.2.1]{Hall1959}, we conclude from the results in Section~\ref{sec_prime_power_order}
that we only have to consider groups with an order that divides $2^4 \cdot 3^2 \cdot 5 \cdot 7$.

In the following we give a summary of the computer search.  
The full list of remaining orders in the sequence that we utilized can be found in Appendix~\ref{sec_app_remainingcompositeorders}.

We considered all remaining orders in the sequence of increasing size.
All conjugacy classes of groups with the orders $6$, $10$, $12$, $14$, $15$, $18$, $21$, $24$, $28$, and $56$
had to be computed. Applying the ILP in Section~\ref{sec_ilp} give that codes larger than $328$ 
are not possible except the group order is  $6$, $12$, or $14$. More precisely, only $G_{6,1}$, $G_{6,2}$, $G_{6,3}$, 
$G_{12,1}$, and  $G_{14,1}$ (cf. Appendix~\ref{sec_app_survivinggroups}) remain. In particular all groups of type $A_4$ were 
excluded, i.e., none of them is an automorphism group of a code of size at least $329$.
The groups of order $36$ were computed but then theoretically excluded since they contain an 
excluded group of prime order or contain a subgroup of type $A_4$.

Next, using the Theorem of Hall \cite[Thm. 9.3.1]{Hall1959} each group of the solvable orders
30, 42, 70, 84, 90, 105, 126, 140, 210, 252, 280, 315, 560, and 630
has a subgroup that was previously excluded.
The groups of order
20, 40, 45, 60, 63, 120, 144, 168, 180, 240, 360, 420, 720, 840, 1008, 1260, and 1680
could be excluded using the Small Groups library~\cite{smallgroupslibrary_homepage}.
The orders 48, 72, 80, 112, 336, and 504
could be excluded along the same lines using a refined analysis, e.g. 
the groups of order 48 contain a subgroup of the excluded order 24 or a subgroup of type $A_4$.
The group orders 35, 2520, and 5040 had to be computed but all of them contain an excluded group of prime order. 
The last two orders, i.e., 2520 and 5040, had to be computed 
because the Hall Theorem~\cite[Thm. 9.3.1]{Hall1959} is not applicable since these orders are non-solvable numbers and 
the Small Groups library does not contain data about groups of these orders.

To sum up, only $G_{6,1}$, $G_{6,2}$, $G_{6,3}$, $G_{12,1}$, and  $G_{14,1}$ (cf. Appendix~\ref{sec_app_survivinggroups}) remain.

\section{Modifying codes to get cardinality 333}\label{sec_modifying_codes}
Since we found an automorphism group of order 16 that yields a code $C$ of size 329, i.e., $G_{16,1}$ in Appendix~\ref{sec_app_survivinggroups}, 
we searched for codes having large intersection with $C$ and automorphism groups $U \le G_{16,1}$.

Therefore, using nonnegative integers $c$ and $c'$, we add the constraint
\[
\sum_{T \in \{t(K,U) \mid K \in C\}} |T\cdot U| \cdot{} x_T \ge c 
\]
to $\text{ILP}(U)$.
This constraint restricts the exchangeability of $U$-orbits.

By choosing the neighborhood parameter $c=300$ and $U=I$, this ILP yielded a code of size $333$, cf. Appendix~\ref{sec_app_333}.
Further investigation showed that the code of size $333$ has the automorphism group $G_{4,6} \le G_{16,1}$ of order $4$,
see Appendix~\ref{sec_app_survivinggroups}.

It turned out that it would have been sufficient to choose $U=G_{4,6}$ and $c=327$ to get a code 
that is extendible to a code of cardinality $333$ having $G_{4,6}$ as automorphism group.
In fact, removing two fixed spaces allows to add two other fixed spaces and two orbits of size two.

$35$ $3$-subspaces of this code of size $333$ are subspaces of the hyperplane in which each vector has zero as first entry. 
Omitting these $35$ subspaces yields a code of size $298$ in the affine geometry $\operatorname{AG}(6,2)$ \cite{zumbragel2016designs}.

\section{Conclusions}\label{sec_conclusion}
In this paper we have considered the problem of the determination of $A_2(7,4;3)$, which is the first open 
case for binary constant dimension codes. Prior to this paper the best known bounds were $329\le A_2(7,4;3)\le 381$. All of 
the previously known constant dimension codes of size $329$ have a trivial automorphism group. By an indirect systematic approach 
we have determined all groups that can be a subgroup of the automorphism group of a constant dimension code in $\mathbb{F}_2^7$ with 
minimum subspace distance $d=4$ that consists of at least $329$ planes. This way we found the unique group of order $16$ that permits 
such a code of size $329$. While not improving the lower bound for the code size, the presence of automorphisms can be beneficial 
in the decoding process. At this place we remark that we are not able to determine the number of conjugacy classes of all subgroups 
of order $16$ in $\operatorname{GL}(\mathbb{F}_2^7)$. Without the systematic approach this group might never have been found. Modifying the 
mentioned code of size $329$ we found a code of cardinality $333$ with an automorphism group of order $4$, which currently is the best known 
construction of a constant dimension code in $\mathbb{F}_2^7$ with minimum subspace distance $4$ and codewords of dimension $3$.

The gap to the upper bound $381$ is still tremendous. However, a lot of effort has been put into the determination of $A_2(7,4;3)$ by 
various researchers. Still the upper bound $381$ can only be excluded for automorphism groups of order larger than $2$. New insights are 
needed to computationally obtain stronger bounds. Our results indicate that, for these specific parameters, good codes either have to have 
small automorphism groups or their size is quite distant to the value of the anticode bound, i.e., $381$. 

In principle the techniques presented in that paper are widely applicable. However, the inherent combinatorial explosion for 
constant dimension codes does not allow too many feasible parameters for not too large groups. For $q=2$ the next open cases are 
$1326\le A_2(8,4;3)\le 1493$ and $4801\le A_2(8,4;4)\le 6477$, see \cite{braun2016new,TableSubspacecodes}. For $A_2(8,4;3)$ e.g.\  the 
group $G_{16,1}$ performs pretty bad and the LP relaxation gives an upper bound of $1292$. Over the ternary field the 
first open case is $754\le A_3(6,4;3)\le 784$, see \cite[Theorem 2]{hkk77} or \cite{cossidente2018geometrical,cossidente2016subspace}. 
Using the systematic approach we were able to reproduce the best known
size 754, but unfortunately no improvement above that has been found.  
First experiments did not yield larger codes than already known in the three parameter sets mentioned above. To get an idea of the 
combinatorial complexity we note that  
the number of solids in $\mathbb{F}_2^8$ is given by $\gauss{8}{4}{2}=200,\!787$. For groups of orders around $20$ the corresponding 
integer linear programs cannot be solved exactly by standard solvers in reasonable time. Even the exclusion of the existence of $381$ planes 
in $\mathbb{F}_2^7$ with minimum subspace distance $4$ that admit an automorphism of order $2$ is currently out of reach \cite{kiermaier2016order}.      

We have applied the presented algorithmic approach to a closely related combinatorial structure. A $t$-$(v,k,\lambda)_q$ packing design 
is a set of $k$-dimensional subspaces of $\mathbb{F}_q^v$ such that every $t$-dimensional subspace is covered at most $\lambda$ times. 
The $2$-$(6,3,2)_2$ packing design of cardinality $180$ with an automorphism group of order $9$ from \cite{qgdd} was quickly rediscovered 
using the presented algorithmic approach. The packing design is indeed optimal, which can be shown using a Johnson-type argument. For 
$2$-$(7,3,\lambda)_2$ packing designs the cardinality is upper bounded by $\lambda\gauss{7}{2}{2}/\gauss{3}{2}{2}=381\lambda$. If the 
upper bound is attained we have a design. For $\lambda=3$ such a design exists, see \cite{Braun2005}, and for $\lambda=1$ the maximum 
cardinality equals $A_2(7,4;3)$. Using our algorithmic approach we found a group of order $27$, isomorphic to the Heisenberg group over 
$\mathbb{F}_3$, that admits a $2$-$(7,3,2)_2$ 
packing design of cardinality $741$, i.e., just $21$ away from the upper bound $762$.
For $2$-$(6,3,3)_3$ packing designs we found an example of cardinality $2368>2262=3\cdot 754$ using a group of order $13^2$.           

The presented algorithmic approach is applicable for a much wider class of combinatorial objects. The only requirements 
are that $\mathcal{P}$ is constant on conjugacy classes and monotone as defined in Section~\ref{sec_preliminaries}. In 
\cite{heinlein2018generalized} the method was applied to find sets of $m_4$ solids and $m_3$ planes in $\mathbb{F}_2^7$ such that 
every plane is covered at most once.


\appendix
\section{The remaining non-prime-power orders}\label{sec_app_remainingcompositeorders}

As stated in Section~\ref{sec_composite_order}, we list here all non-prime-power numbers 
which divide $2^4 \cdot 3^2 \cdot 5 \cdot 7$. 
They have to be considered as size of a subgroup in the group $\operatorname{GL}(\mathbb{F}_2^7)$ 
to determine an exhaustive list of groups such that no other group of non-prime-power order than these listed here 
is an automorphism group of a code of size at least $329$.
In parentheses we note the line of reasoning: 
``Small Groups library'' means that the abstract type is used to show the existence of already excluded subgroups.
``Hall, solvable order'' means that the Theorem of Hall~\cite[Theorem 9.3.1]{Hall1959} is used to show 
the existence of already excluded subgroups.
Moreover ``due to groups of prime order'' means that the group has a subgroup that is excluded within Section~\ref{sec_prime_power_order}.

\begin{itemize}
\item[6] there are 12 subgroups of order 6 up to conjugacy in the group $\operatorname{GL}(\mathbb{F}_2^7)$. 9 are excluded due to groups of prime order. The 3 remaining groups cannot be excluded.
\item[10] there are 3 subgroups of order 10 up to conjugacy in the group $\operatorname{GL}(\mathbb{F}_2^7)$. 2 are excluded due to groups of prime order. The remaining group yields codes of size up to 306.
\item[12] there are 96 subgroups of order 12 up to conjugacy in the group $\operatorname{GL}(\mathbb{F}_2^7)$. 80 are excluded due to groups of prime order. All but 1 group could be excluded, it is of type $Z_{3} \rtimes Z_{4}$.
\item[14] there are 4 subgroups of order 14 up to conjugacy in the group $\operatorname{GL}(\mathbb{F}_2^7)$. 2 are excluded due to groups of prime order. One could be excluded and the other yields codes of size at most 332. The remaining group is of abstract type $Z_{14}$. One of these two groups could be solved in less then 60 seconds with an optimal value of 301. The other one was much harder and the technique described in Subsection~\ref{subsec_symmetry_normalizer} was applied. The orbit type is $1^1 2^4 7^{30} 14^{828}$ and after removing the trivially forbidden orbits $1^1 2^4 7^{28} 14^{632}$. The normalizer has order 168 and the normalizer-orbit type is $1^1 4^{13} 6^2 12^{50}$ making a total of $66$ subproblems.
\item[15] there are 3 subgroups of order 15 up to conjugacy in the group $\operatorname{GL}(\mathbb{F}_2^7)$. 1 is excluded due to groups of prime order. The remaining groups could be excluded.
\item[18] there are 16 subgroups of order 18 up to conjugacy in the group $\operatorname{GL}(\mathbb{F}_2^7)$. 13 are excluded due to groups of prime order. The remaining groups could be excluded.
\item[20] each group of order 20 contains a group of order 10 (Small Groups library)
\item[21] there are 8 subgroups of order 21 up to conjugacy in the group $\operatorname{GL}(\mathbb{F}_2^7)$. 5 are excluded due to groups of prime order. The remaining groups could be excluded.
\item[24] there are 525 subgroups of order 24 up to conjugacy in the group $\operatorname{GL}(\mathbb{F}_2^7)$. 488 are excluded due to groups of prime order. The types of these groups are: $14$ times $S_{4}$, $19$ times $Z_{2} \times A_{4}$, $2$ times $\operatorname{SL}(2,3)$, and 2 times $(Z_{6} \times Z_{2}) \rtimes Z_{2}$. All but the two groups of type $\operatorname{SL}(2,3)$ contain an excluded $Z_{12}$, $Z_{6} \times Z_{2}$, or $A_{4}$. The remaining two groups could be excluded computationally.
\item[28] there are 9 subgroups of order 28 up to conjugacy in the group $\operatorname{GL}(\mathbb{F}_2^7)$. 8 are excluded due to groups of prime order. The remaining group is of type $Z_{14} \times Z_{2}$ but could be excluded computationally.
\item[30] each group of order 30 contains a group of order 10 (Hall, solvable order)
\item[35] there is 1 subgroup of order 35 up to conjugacy in the group $\operatorname{GL}(\mathbb{F}_2^7)$. It is excluded due to groups of prime order.
\item[36] there are 61 subgroups of order 36 up to conjugacy in the group $\operatorname{GL}(\mathbb{F}_2^7)$. 59 are excluded due to groups of prime order. The remaining groups are both of type $Z_{3} \times A_{4}$ and contain an excluded $A_{4}$.
\item[40] each group of order~$40$ contains a group of order~$10$ (Small Groups library)
\item[42] each group of order~$42$ contains a group of order~$21$ (Hall, solvable order)
\item[45] each group of order~$45$ contains a group of order~$15$ (Small Groups library)
\item[48] each group of order~$48$ contains a subgroup of order~$24$ or a subgroup of abstract type $A_{4}$ (Small Groups library)
\item[56] there are $38$ subgroups of order~$56$ up to conjugacy in the group $\operatorname{GL}(\mathbb{F}_2^7)$. 26 are excluded due to groups of prime order. One group is of type $Z_{14} \times Z_{2} \times Z_{2}$ and contains an excluded $Z_{14}$. The remaining 11 groups are of type $Z_{2} \times Z_{2} \times Z_{2} \times Z_{7}$ but could be excluded computationally.
\item[60] each group of order~$60$ contains a group of order~$10$ (Small Groups library)
\item[63] each group of order~$63$ contains a group of order~$21$ (Small Groups library)
\item[70] each group of order~$70$ contains a group of order~$10$ (Hall, solvable order)
\item[72] each group of order~$72$ contains a group of order~$36$ or a subgroup of abstract type $Z_{12}$ (Small Groups library)
\item[80] each group of order~$80$ contains a subgroup of order~$10$ or a subgroup of abstract type $Z_{2} \times Z_{2} \times Z_{2} \times Z_{2}$, which yields codes of size at most 313 (Small Groups library)
\item[84] each group of order~$84$ contains a group of order~$28$ (Hall, solvable order)
\item[90] each group of order~$90$ contains a group of order~$10$ (Hall, solvable order)
\item[105] each group of order~$105$ contains a group of order~$15$ (Hall, solvable order)
\item[112] each group of order~$112$ contains a subgroup of order~$28$ or a subgroup of abstract type $Z_{2} \times Z_{2} \times Z_{2} \times Z_{2}$ (Small Groups library)
\item[120] each group of order~$120$ contains a group of order~$10$ (Small Groups library)
\item[126] each group of order~$126$ contains a group of order~$63$ (Hall, solvable order)
\item[140] each group of order~$140$ contains a group of order~$28$ (Hall, solvable order)
\item[144] each group of order~$144$ contains a group of order~$36$ (Small Groups library)
\item[168] each group of order~$168$ contains a group of order~$21$ (Small Groups library)
\item[180] each group of order~$180$ contains a group of order~$36$ (Small Groups library)
\item[210] each group of order~$210$ contains a group of order~$10$ (Hall, solvable order)
\item[240] each group of order~$240$ contains a group of order~$10$ or order~$15$ (Small Groups library)
\item[252] each group of order~$252$ contains a group of order~$28$ (Hall, solvable order)
\item[280] each group of order~$280$ contains a group of order~$35$ (Hall, solvable order)
\item[315] each group of order~$315$ contains a group of order~$63$ (Hall, solvable order)
\item[336] each group of order~$336$ contains a subgroup of order~$48$ or a subgroup of abstract type $A_{4}$ or $Q_{16}$ (Small Groups library)
\item[360] each group of order~$360$ contains a group of order~$10$ (Small Groups library)
\item[420] each group of order~$420$ contains a group of order~$28$ (Small Groups library)
\item[504] each group of order~$504$ contains a subgroup of order~$63$ or a subgroup of abstract type $D_{14}$ (Small Groups library)
\item[560] each group of order~$560$ contains a group of order~$35$ (Hall, solvable order)
\item[630] each group of order~$630$ contains a group of order~$10$ (Hall, solvable order)
\item[720] each group of order~$720$ contains a group of order~$10$ or order~$45$ (Small Groups library)
\item[840] each group of order~$840$ contains a group of order~$10$ (Small Groups library)
\item[1008] each group of order~$1008$ contains a group of order~$36$ or order~$63$ (Small Groups library)
\item[1260] each group of order~$1260$ contains a group of order~$10$ (Small Groups library)
\item[1680] each group of order~$1680$ contains a group of order~$10$ or order~$15$ (Small Groups library)
\item[2520] there are 7 subgroups of order~$2520$ up to conjugacy in the group $\operatorname{GL}(\mathbb{F}_2^7)$. All are excluded due to groups of prime order.
\item[5040] there are 4 subgroups of order~$5040$ up to conjugacy in the group $\operatorname{GL}(\mathbb{F}_2^7)$. All are excluded due to groups of prime order. None of them is solvable.
\end{itemize}

\section{The \textit{surviving} groups}\label{sec_app_survivinggroups}
By $G_{n,m}$ we denote the groups corresponding to Theorem~\ref{main_thm_1}. Here $n$ denotes the order of $G_{n,m}$ and $m$ is a consecutive index.
To the right of each group $G_{n,m}$ we list the abstract type of $G_{n,m}$.
\tiny

\begin{flalign*}
G_{1,1} = I
&&Z_{1}\end{flalign*}

\begin{flalign*}
G_{2,1}
=
\left\langle
\left(\begin{smallmatrix}
1&0&0&0&0&0&0\\
1&1&0&0&0&0&0\\
0&0&1&0&0&0&0\\
0&0&1&1&0&0&0\\
0&0&0&0&1&0&0\\
0&0&0&0&1&1&0\\
0&0&0&0&0&0&1\\
\end{smallmatrix}\right)
\right\rangle
&&Z_{2}\end{flalign*}

\begin{flalign*}
G_{3,1}
=
\left\langle
\left(\begin{smallmatrix}
1&1&0&0&0&0&0\\
1&0&0&0&0&0&0\\
0&0&1&1&0&0&0\\
0&0&1&0&0&0&0\\
0&0&0&0&1&1&0\\
0&0&0&0&1&0&0\\
0&0&0&0&0&0&1\\
\end{smallmatrix}\right)
\right\rangle
&&Z_{3}\end{flalign*}

\begin{flalign*}
G_{3,2}
=
\left\langle
\left(\begin{smallmatrix}
1&1&0&0&0&0&0\\
1&0&0&0&0&0&0\\
0&0&1&1&0&0&0\\
0&0&1&0&0&0&0\\
0&0&0&0&1&0&0\\
0&0&0&0&0&1&0\\
0&0&0&0&0&0&1\\
\end{smallmatrix}\right)
\right\rangle
&&Z_{3}\end{flalign*}

\begin{flalign*}
G_{4,1}
=
\left\langle
\left(\begin{smallmatrix}
0&0&1&0&1&0&0\\
0&1&1&0&0&1&0\\
0&0&1&1&0&0&0\\
0&0&0&1&0&0&0\\
1&0&1&1&0&0&0\\
0&0&0&1&0&1&0\\
0&0&0&0&0&0&1
\end{smallmatrix}\right)
,
\left(\begin{smallmatrix}
0&0&0&1&1&1&1\\
0&1&1&1&0&1&0\\
1&0&0&1&1&0&1\\
0&0&0&1&0&0&0\\
0&0&1&0&1&1&0\\
1&0&1&1&1&1&1\\
0&0&0&0&0&0&1
\end{smallmatrix}\right)
\right\rangle
&&Z_{2} \times Z_{2}\end{flalign*}

\begin{flalign*}
G_{4,2}
=
\left\langle
\left(\begin{smallmatrix}
1&0&0&1&0&1&0\\
1&0&0&1&1&1&1\\
0&0&1&1&0&1&0\\
1&0&1&1&0&0&0\\
1&0&1&0&1&0&0\\
1&0&1&0&0&1&0\\
0&1&1&0&1&0&0\\
\end{smallmatrix}\right)
,
\left(\begin{smallmatrix}
0&1&0&0&0&1&1\\
1&1&1&1&0&1&0\\
1&1&1&0&0&1&1\\
1&0&1&1&1&1&0\\
1&0&1&0&0&1&0\\
1&0&1&0&1&0&0\\
1&1&0&1&1&1&0\\
\end{smallmatrix}\right)
\right\rangle
&&Z_{2} \times Z_{2}\end{flalign*}

\begin{flalign*}
G_{4,3}
=
\left\langle
\left(\begin{smallmatrix}
1&1&0&0&0&1&0\\
1&1&1&1&1&0&0\\
0&1&1&0&0&1&0\\
1&0&1&0&1&0&0\\
1&0&1&1&0&0&0\\
1&0&1&1&1&1&0\\
0&1&0&1&1&1&1\\
\end{smallmatrix}\right)
,
\left(\begin{smallmatrix}
0&0&1&1&1&0&0\\
1&0&1&1&1&1&0\\
1&0&0&1&1&0&0\\
1&1&1&1&0&1&0\\
1&1&1&0&1&1&0\\
1&1&1&1&1&0&0\\
1&0&1&1&1&0&1\\
\end{smallmatrix}\right)
\right\rangle
&&Z_{2} \times Z_{2}\end{flalign*}

\begin{flalign*}
G_{4,4}
=
\left\langle
\left(\begin{smallmatrix}
1&1&1&0&0&1&1\\
1&1&1&1&0&1&0\\
0&1&0&1&1&1&1\\
1&1&0&1&1&0&1\\
1&1&0&0&0&0&1\\
1&1&0&1&0&1&1\\
0&0&0&1&1&0&1\\
\end{smallmatrix}\right)
,
\left(\begin{smallmatrix}
0&0&1&0&1&1&0\\
0&0&1&0&0&1&1\\
1&0&0&1&0&1&0\\
0&1&1&1&0&1&1\\
0&1&1&0&1&1&1\\
0&1&1&1&1&0&1\\
1&0&1&0&1&1&1\\
\end{smallmatrix}\right)
\right\rangle
&&Z_{2} \times Z_{2}\end{flalign*}

\begin{flalign*}
G_{4,5}
=
\left\langle
\left(\begin{smallmatrix}
1&1&1&0&1&0&1\\
0&0&1&0&0&1&1\\
0&1&0&1&0&0&1\\
0&0&0&1&0&0&0\\
0&0&0&0&1&0&0\\
0&0&0&0&0&1&0\\
0&0&0&1&0&1&1\\
\end{smallmatrix}\right)
,
\left(\begin{smallmatrix}
0&0&1&0&1&1&0\\
0&0&1&0&0&1&1\\
1&0&0&1&0&1&0\\
0&1&1&1&0&1&1\\
0&1&1&0&1&1&1\\
0&1&1&1&1&0&1\\
1&0&1&0&1&1&1\\
\end{smallmatrix}\right)
\right\rangle
&&Z_{2} \times Z_{2}\end{flalign*}

\begin{flalign*}
G_{4,6} 
=
\left\langle
\left(\begin{smallmatrix}
0&0&1&0&0&0&0\\
0&0&0&0&1&0&0\\
1&0&0&0&0&0&0\\
0&1&0&1&1&0&0\\
0&1&0&0&0&0&0\\
0&1&0&0&1&1&0\\
1&0&1&1&1&0&1\\
\end{smallmatrix}\right)
,
\left(\begin{smallmatrix}
1&1&0&0&0&1&0\\
1&1&0&1&1&0&1\\
0&1&1&0&0&1&0\\
1&0&0&1&0&0&1\\
1&0&0&0&1&0&1\\
1&0&0&1&1&1&1\\
0&1&0&0&0&1&1\\
\end{smallmatrix}\right)
\right\rangle
&&Z_{2} \times Z_{2}\end{flalign*}

\begin{flalign*}
G_{4,7}
=
\left\langle
\left(\begin{smallmatrix}
1&1&0&0&0&0&0\\
0&1&1&0&0&0&0\\
0&0&1&0&0&0&0\\
0&0&0&1&1&0&0\\
0&0&0&0&1&1&0\\
0&0&0&0&0&1&1\\
0&0&0&0&0&0&1\\
\end{smallmatrix}\right)
\right\rangle
&&Z_{4}\end{flalign*}

\begin{flalign*}
G_{5,1}
=
\left\langle
\left(\begin{smallmatrix}
0&1&0&0&0&0&0\\
0&0&1&0&0&0&0\\
0&0&0&1&0&0&0\\
1&1&1&1&0&0&0\\
0&0&0&0&1&0&0\\
0&0&0&0&0&1&0\\
0&0&0&0&0&0&1\\
\end{smallmatrix}\right)
\right\rangle
&&Z_{5}\end{flalign*}

\begin{flalign*}
G_{6,1}
=
\left\langle
\left(\begin{smallmatrix}
0&1&0&0&1&1&0\\
1&1&0&0&0&1&0\\
0&1&1&1&1&0&0\\
0&0&0&1&0&0&0\\
1&0&0&0&1&1&0\\
1&1&0&0&1&0&0\\
0&0&0&0&0&0&1
\end{smallmatrix}\right)
,
\left(\begin{smallmatrix}
0&1&1&1&1&0&0\\
1&1&1&1&0&0&0\\
0&1&0&0&1&1&0\\
0&1&1&1&0&1&0\\
0&0&1&1&0&0&0\\
1&0&1&0&1&1&0\\
0&0&0&0&0&0&1
\end{smallmatrix}\right)
\right\rangle
&&S_{3}\end{flalign*}

\begin{flalign*}
G_{6,2}
=
\left\langle
\left(\begin{smallmatrix}
1&1&0&1&0&1&0\\
1&1&0&1&1&0&0\\
0&0&0&1&1&1&0\\
1&1&1&0&0&0&0\\
1&0&1&0&1&1&0\\
0&1&1&0&1&1&0\\
0&0&0&0&0&0&1
\end{smallmatrix}\right)
,
\left(\begin{smallmatrix}
1&0&1&1&0&1&0\\
1&0&1&0&1&0&0\\
0&0&0&0&1&1&0\\
1&1&0&0&1&1&0\\
1&0&0&0&0&0&0\\
0&1&1&1&1&1&0\\
0&0&0&0&0&0&1
\end{smallmatrix}\right)
\right\rangle
&&S_{3}\end{flalign*}

\begin{flalign*}
G_{6,3}
=
\left\langle
\left(\begin{smallmatrix}
1&0&0&0&1&0&0\\
1&0&0&1&1&1&0\\
1&1&1&0&1&0&0\\
1&0&1&1&1&0&0\\
1&1&1&0&0&0&0\\
0&1&1&0&1&1&0\\
0&0&0&0&0&0&1
\end{smallmatrix}\right)
\right\rangle
&&Z_{6}\end{flalign*}

\begin{flalign*}
G_{7,1}
=
\left\langle
\left(\begin{smallmatrix}
0&1&0&0&0&0&0\\
0&0&1&0&0&0&0\\
1&0&1&0&0&0&0\\
0&0&0&0&1&0&0\\
0&0&0&0&0&1&0\\
0&0&0&1&1&0&0\\
0&0&0&0&0&0&1\\
\end{smallmatrix}\right)
\right\rangle
&&Z_{7}\end{flalign*}

\begin{flalign*}
G_{7,2}
=
\left\langle
\left(\begin{smallmatrix}
0&1&0&0&0&0&0\\
0&0&1&0&0&0&0\\
1&0&1&0&0&0&0\\
0&0&0&0&1&0&0\\
0&0&0&0&0&1&0\\
0&0&0&1&0&1&0\\
0&0&0&0&0&0&1\\
\end{smallmatrix}\right)
\right\rangle
&&Z_{7}\end{flalign*}

\begin{flalign*}
G_{8,1}
=
\left\langle
\left(\begin{smallmatrix}
1&1&1&0&1&0&0\\
1&0&0&0&0&1&0\\
0&0&0&1&0&0&0\\
0&0&1&0&0&0&0\\
1&1&1&1&1&1&0\\
1&0&1&0&1&0&0\\
0&0&0&0&0&0&1\\
\end{smallmatrix}\right)
,
\left(\begin{smallmatrix}
0&0&1&0&1&1&1\\
1&0&0&1&0&0&0\\
1&1&0&0&0&0&1\\
0&1&1&0&1&1&1\\
0&0&1&1&1&0&1\\
1&1&1&0&0&1&1\\
1&0&0&1&1&1&1\\
\end{smallmatrix}\right)
,
\left(\begin{smallmatrix}
0&1&0&0&0&1&1\\
1&1&1&0&0&1&1\\
1&0&0&0&0&1&1\\
0&1&0&0&0&0&1\\
0&0&0&0&1&0&0\\
1&0&0&1&0&0&0\\
1&1&1&1&0&1&1\\
\end{smallmatrix}\right)
\right\rangle
&&Z_{2} \times Z_{2} \times Z_{2}\end{flalign*}

\begin{flalign*}
G_{8,2}
=
\left\langle
\left(\begin{smallmatrix}
1&0&0&1&1&1&1\\
0&0&1&1&0&1&1\\
0&0&0&0&1&1&0\\
0&1&1&1&1&0&0\\
0&1&0&1&0&0&1\\
0&1&1&1&0&0&1\\
0&1&0&0&0&1&1\\
\end{smallmatrix}\right)
,
\left(\begin{smallmatrix}
1&0&1&0&1&1&0\\
0&0&1&0&1&0&0\\
0&0&0&1&0&0&1\\
0&0&0&0&1&1&1\\
0&1&0&1&0&0&1\\
0&1&1&0&1&1&0\\
0&0&1&0&1&1&1\\
\end{smallmatrix}\right)
,
\left(\begin{smallmatrix}
1&0&1&1&1&0&1\\
1&0&0&0&0&1&1\\
1&1&1&0&0&1&1\\
1&0&1&1&0&1&1\\
0&0&0&0&1&0&0\\
1&0&0&1&1&0&0\\
0&1&1&0&0&0&1\\
\end{smallmatrix}\right)
\right\rangle
&&Z_{2} \times Z_{2} \times Z_{2}\end{flalign*}

\begin{flalign*}
G_{8,3}
=
\left\langle
\left(\begin{smallmatrix}
1&0&1&1&0&0&0\\
1&0&0&0&1&1&1\\
1&0&0&0&0&1&1\\
0&0&0&1&1&0&0\\
0&1&1&0&0&0&0\\
0&0&0&0&1&1&1\\
1&0&1&0&1&1&0\\
\end{smallmatrix}\right)
,
\left(\begin{smallmatrix}
1&0&1&1&0&0&1\\
1&0&0&0&0&1&1\\
1&0&1&1&1&1&0\\
1&0&0&0&1&1&0\\
0&1&0&1&0&0&1\\
1&1&1&1&0&0&0\\
0&0&0&0&0&0&1\\
\end{smallmatrix}\right)
,
\left(\begin{smallmatrix}
0&0&1&0&1&1&1\\
0&1&1&1&0&0&1\\
0&0&0&1&0&0&1\\
0&1&1&1&1&0&0\\
0&0&0&0&1&0&0\\
1&1&0&0&0&0&1\\
0&1&0&1&1&0&0\\
\end{smallmatrix}\right)
\right\rangle
&&Z_{4} \times Z_{2}\end{flalign*}

\begin{flalign*}
G_{8,4}
=
\left\langle
\left(\begin{smallmatrix}
1&0&0&1&1&1&1\\
1&1&0&1&0&1&0\\
1&1&1&0&0&1&1\\
1&1&0&0&0&0&0\\
1&0&1&0&0&1&1\\
0&0&1&1&0&1&1\\
0&1&0&0&0&1&1\\
\end{smallmatrix}\right)
,
\left(\begin{smallmatrix}
0&0&0&1&1&1&0\\
1&1&0&0&0&0&1\\
0&0&0&0&0&1&0\\
0&1&0&1&1&1&0\\
0&0&0&0&1&0&0\\
1&1&0&1&1&1&0\\
0&1&1&0&0&0&1\\
\end{smallmatrix}\right)
,
\left(\begin{smallmatrix}
1&0&0&0&1&0&0\\
1&0&1&1&1&1&0\\
1&1&0&1&1&1&0\\
0&1&0&0&1&0&1\\
0&0&0&0&1&0&0\\
0&1&0&1&0&1&1\\
1&0&1&0&0&1&0\\
\end{smallmatrix}\right)
\right\rangle
&&Q_{8}\end{flalign*}

\begin{flalign*}
G_{8,5}
=
\left\langle
\left(\begin{smallmatrix}
0&0&1&0&0&1&1\\
1&1&0&0&0&0&0\\
0&0&0&0&1&1&1\\
1&0&1&0&1&0&1\\
1&1&1&1&1&1&0\\
1&0&1&1&0&1&1\\
1&1&0&0&1&1&0\\
\end{smallmatrix}\right)
,
\left(\begin{smallmatrix}
0&0&0&1&1&1&0\\
0&1&1&1&1&0&0\\
0&1&0&0&0&0&1\\
0&1&1&1&1&0&1\\
1&0&1&0&0&1&1\\
1&0&1&0&1&0&0\\
1&1&1&1&1&1&1\\
\end{smallmatrix}\right)
,
\left(\begin{smallmatrix}
0&1&1&1&0&1&0\\
1&1&1&0&0&1&1\\
1&0&0&0&0&1&1\\
0&0&0&1&1&0&0\\
0&0&0&0&1&0&0\\
1&1&1&1&1&0&0\\
1&0&1&0&1&1&0\\
\end{smallmatrix}\right)
\right\rangle
&&Q_{8}\end{flalign*}

\begin{flalign*}
G_{8,6}
=
\left\langle
\left(\begin{smallmatrix}
1&1&0&1&1&0&1\\
1&1&1&1&1&0&1\\
1&1&0&0&0&0&0\\
0&1&0&1&0&1&1\\
0&1&0&1&0&0&1\\
0&0&0&1&1&0&0\\
1&0&1&0&0&1&0\\
\end{smallmatrix}\right)
,
\left(\begin{smallmatrix}
0&0&0&1&1&1&0\\
0&1&0&1&1&1&1\\
0&0&0&0&1&1&0\\
1&0&1&0&0&0&0\\
0&0&1&1&1&0&1\\
0&0&0&1&1&0&1\\
1&0&0&1&1&1&1\\
\end{smallmatrix}\right)
,
\left(\begin{smallmatrix}
1&1&0&1&1&0&1\\
0&0&1&0&0&0&0\\
0&1&0&0&0&0&0\\
0&0&0&1&1&0&0\\
0&0&0&0&1&0&0\\
0&1&0&1&0&1&1\\
0&1&1&0&1&0&1\\
\end{smallmatrix}\right)
\right\rangle
&&D_{8}\end{flalign*}

\begin{flalign*}
G_{8,7}
=
\left\langle
\left(\begin{smallmatrix}
0&0&1&0&0&1&1\\
1&1&0&0&0&0&0\\
1&1&0&0&1&0&0\\
1&0&0&1&1&0&0\\
0&1&1&0&0&0&0\\
0&0&0&1&1&0&0\\
0&1&1&0&0&0&1\\
\end{smallmatrix}\right)
,
\left(\begin{smallmatrix}
0&1&0&0&0&1&1\\
1&1&0&0&0&0&1\\
1&1&1&1&1&0&0\\
0&0&0&0&1&1&1\\
0&1&0&1&0&0&1\\
1&0&0&0&0&1&1\\
1&1&0&0&0&1&0\\
\end{smallmatrix}\right)
,
\left(\begin{smallmatrix}
1&0&1&1&0&0&1\\
1&1&1&0&0&1&1\\
1&0&0&0&0&1&1\\
1&1&0&1&1&1&1\\
0&0&0&0&1&0&0\\
1&1&1&1&1&0&0\\
0&1&1&0&1&0&1\\
\end{smallmatrix}\right)
\right\rangle
&&Z_{4} \times Z_{2}\end{flalign*}

\begin{flalign*}
G_{8,8}
=
\left\langle
\left(\begin{smallmatrix}
1&0&0&1&1&1&1\\
1&1&0&1&0&1&0\\
1&1&1&0&0&1&1\\
1&1&0&0&0&0&0\\
1&0&1&0&0&1&1\\
0&0&1&1&0&1&1\\
0&1&0&0&0&1&1\\
\end{smallmatrix}\right)
,
\left(\begin{smallmatrix}
0&0&0&1&0&1&0\\
0&1&0&1&1&1&1\\
1&0&0&1&1&0&0\\
1&1&0&0&0&0&0\\
0&0&0&0&1&0&0\\
0&1&0&0&1&0&0\\
0&1&1&0&0&0&1\\
\end{smallmatrix}\right)
,
\left(\begin{smallmatrix}
1&0&0&0&1&0&0\\
1&0&1&1&1&1&0\\
1&1&0&1&1&1&0\\
0&1&0&0&1&0&1\\
0&0&0&0&1&0&0\\
0&1&0&1&0&1&1\\
1&0&1&0&0&1&0\\
\end{smallmatrix}\right)
\right\rangle
&&Z_{4} \times Z_{2}\end{flalign*}

\begin{flalign*}
G_{8,9}
=
\left\langle
\left(\begin{smallmatrix}
0&0&1&0&0&1&1\\
1&1&1&0&0&1&1\\
1&1&0&1&1&1&0\\
1&0&0&0&0&1&0\\
0&1&0&1&0&0&1\\
0&0&1&1&0&1&1\\
0&1&1&0&1&0&1\\
\end{smallmatrix}\right)
,
\left(\begin{smallmatrix}
0&1&1&0&0&0&1\\
0&1&0&1&0&1&1\\
0&0&0&0&0&1&0\\
1&1&1&0&0&1&0\\
0&0&1&1&1&0&1\\
0&0&1&0&0&0&0\\
1&1&0&1&0&0&1\\
\end{smallmatrix}\right)
,
\left(\begin{smallmatrix}
0&0&0&1&0&1&0\\
1&0&0&0&1&1&1\\
1&1&1&0&1&1&1\\
0&1&1&1&1&0&0\\
0&0&0&0&1&0&0\\
1&1&1&1&1&0&0\\
1&0&1&0&0&1&0\\
\end{smallmatrix}\right)
\right\rangle
&&D_{8}\end{flalign*}

\begin{flalign*}
G_{8,10}
=
\left\langle
\left(\begin{smallmatrix}
1&0&0&0&0&0&0\\
0&0&1&0&1&0&0\\
0&0&0&1&0&0&1\\
1&0&0&0&0&1&0\\
0&1&0&1&0&0&1\\
1&0&0&1&0&0&0\\
1&0&1&0&0&1&0\\
\end{smallmatrix}\right)
,
\left(\begin{smallmatrix}
1&0&0&1&0&1&1\\
1&1&0&0&1&0&1\\
1&0&0&1&1&0&0\\
1&0&0&0&1&1&0\\
0&0&1&1&1&0&1\\
1&0&1&1&1&1&0\\
0&0&1&0&0&1&1\\
\end{smallmatrix}\right)
,
\left(\begin{smallmatrix}
0&0&0&1&0&1&0\\
1&0&0&0&1&1&1\\
1&1&1&0&1&1&1\\
0&1&1&1&1&0&0\\
0&0&0&0&1&0&0\\
1&1&1&1&1&0&0\\
1&0&1&0&0&1&0\\
\end{smallmatrix}\right)
\right\rangle
&&D_{8}\end{flalign*}

\begin{flalign*}
G_{8,11}
=
\left\langle
\left(\begin{smallmatrix}
0&0&1&1&1&0&0\\
0&1&0&0&0&0&1\\
0&0&1&0&1&0&1\\
1&1&0&0&1&0&1\\
0&1&0&1&0&0&1\\
1&1&1&1&1&0&1\\
0&1&0&0&0&1&1\\
\end{smallmatrix}\right)
\right\rangle
&&Z_{8}\end{flalign*}

\begin{flalign*}
G_{9,1}
=
\left\langle
\left(\begin{smallmatrix}
1&0&1&1&0&1&0\\
1&0&1&1&1&0&0\\
0&0&1&1&1&0&0\\
1&1&0&0&1&1&0\\
1&1&0&1&1&0&0\\
0&1&0&0&0&1&0\\
0&0&0&0&0&0&1\\
\end{smallmatrix}\right)
\right\rangle
&&Z_{9}\end{flalign*}

\begin{flalign*}
G_{9,2}
=
\left\langle
\left(\begin{smallmatrix}
0&1&0&0&0&1&0\\
1&0&0&0&0&1&0\\
0&1&1&0&0&1&0\\
1&1&0&0&1&0&0\\
1&1&0&1&1&0&0\\
0&1&0&0&0&0&0\\
0&0&0&0&0&0&1
\end{smallmatrix}\right)
,
\left(\begin{smallmatrix}
0&1&0&0&0&1&0\\
1&0&1&0&0&0&0\\
1&0&1&0&0&1&0\\
0&1&1&1&1&0&0\\
0&0&0&1&0&1&0\\
0&1&1&0&0&1&0\\
0&0&0&0&0&0&1
\end{smallmatrix}\right)
\right\rangle
&&Z_{3} \times Z_{3}\end{flalign*}

\begin{flalign*}
G_{12,1}
=
\left\langle
\left(\begin{smallmatrix}
1&0&0&0&0&1&1\\
0&0&0&1&1&0&1\\
1&1&1&1&1&0&0\\
1&1&0&0&1&1&0\\
0&0&0&0&0&0&1\\
0&0&0&0&1&1&1\\
0&0&0&0&1&0&0
\end{smallmatrix}\right)
,
\left(\begin{smallmatrix}
1&0&0&0&0&0&0\\
1&1&0&0&0&1&1\\
1&0&1&0&1&0&1\\
1&0&0&1&0&0&0\\
0&0&0&0&1&0&0\\
0&0&0&0&0&1&0\\
0&0&0&0&0&0&1
\end{smallmatrix}\right)
,
\left(\begin{smallmatrix}
1&0&0&0&0&1&1\\
0&1&0&1&1&1&1\\
1&0&1&1&1&0&0\\
1&1&0&0&0&1&1\\
1&0&0&0&1&0&0\\
1&0&0&0&0&1&0\\
0&0&0&0&0&1&0
\end{smallmatrix}\right)
\right\rangle
&&Z_{3} \rtimes Z_{4}\end{flalign*}

\begin{flalign*}
G_{14,1}
=
\left\langle
\left(\begin{smallmatrix}
0&1&1&1&1&0&0 \\
0&1&1&0&0&0&0 \\
0&1&1&0&1&0&0 \\
0&1&0&0&0&0&0 \\
0&0&1&0&1&1&0 \\
1&0&1&0&0&1&0 \\
0&0&0&0&0&0&1
\end{smallmatrix}\right)
\right\rangle
&&Z_{14}\end{flalign*}

\begin{flalign*}
G_{16,1}
=
\left\langle
\left(\begin{smallmatrix}
0&0&1&0&1&0&0\\
1&0&0&0&1&0&0\\
0&0&0&1&0&1&0\\
0&1&0&0&0&0&1\\
1&0&1&0&1&1&1\\
1&0&1&0&0&1&0\\
0&0&1&1&1&1&1\\
\end{smallmatrix}\right)
,
\left(\begin{smallmatrix}
0&0&1&1&0&1&1\\
1&0&1&1&1&0&1\\
0&1&1&1&1&1&0\\
0&0&1&1&0&1&0\\
1&1&1&1&0&1&0\\
1&0&0&1&0&1&1\\
0&0&1&0&0&0&0\\
\end{smallmatrix}\right)
\right\rangle
&&(Z_{4} \times Z_{2}) \rtimes Z_{2}\end{flalign*}


\normalsize

\section{The code of size 333 in the binary Fano setting}
\label{sec_app_333}

The code of size 333 is printed below.
Since the group $G_{4,6}$ of Appendix~\ref{sec_app_survivinggroups} is its automorphism group we print only one representative in each orbit.
The orbit type is $1^9 2^{26} 4^{68}$.
Each row represents one subspace and each number represents a column in the reduced row echelon form matrix corresponding to the subspace by multiplying the entries in the column with powers of $2$: $\begin{smallmatrix}a\\b\\c\end{smallmatrix} \leftrightarrow a \cdot 2^0 + b \cdot 2^1 + c \cdot 2^2$.
For example, the first line in the representatives of order~$4$, i.e., $0102004$, is the orbit of subspaces:
\[
\left(
\operatorname{im}
\left(
\begin{smallmatrix}
0&1&0&0&0&0&0\\
0&0&0&1&0&0&0\\
0&0&0&0&0&0&1\\
\end{smallmatrix} \right) \right) \cdot G_{4,6}
\]

9~fixed blocks:\\

{
\footnotesize
\renewcommand{\arraystretch}{0.8}
\setlength{\tabcolsep}{3pt}
\begin{supertabular}{lllllll}
0&1&2&4&4&1&2\\
1&0&1&2&4&6&0\\
1&1&2&4&6&3&3\\
1&2&0&4&6&0&1\\
1&2&1&3&4&5&7\\
1&2&1&4&4&2&5\\
1&2&2&4&7&1&3\\
1&2&4&0&0&2&0\\
1&2&4&2&7&7&0\\
\end{supertabular}
}\\

26~representatives of orbits of length 2:\\

{
\footnotesize
\renewcommand{\arraystretch}{0.8}
\setlength{\tabcolsep}{3pt}
\begin{supertabular}{lllllll}
0&1&0&2&1&4&0\\
1&0&2&4&4&5&3\\
1&1&1&2&4&3&4\\
1&1&2&2&1&2&4\\
1&1&2&3&3&4&6\\
1&2&0&4&5&7&1\\
1&2&1&0&4&1&0\\
1&2&1&1&4&6&0\\
1&2&1&2&4&7&3\\
1&2&1&4&3&3&6\\
1&2&3&0&4&2&6\\
1&2&4&1&1&1&6\\
1&2&4&2&3&7&5\\
1&2&4&2&4&1&5\\
1&2&4&2&5&7&7\\
1&2&4&3&3&4&5\\
1&2&4&3&4&2&2\\
1&2&4&3&7&7&4\\
1&2&4&4&1&0&5\\
1&2&4&4&1&6&4\\
1&2&4&4&2&2&5\\
1&2&4&5&1&3&0\\
1&2&4&5&3&4&6\\
1&2&4&5&5&0&5\\
1&2&4&5&7&7&5\\
1&2&4&6&3&5&7\\
\end{supertabular}
}\\

68~representatives of orbits of length 4:\\

{
\footnotesize
\renewcommand{\arraystretch}{0.8}
\setlength{\tabcolsep}{3pt}
\begin{supertabular}{lllllll}
0&1&0&2&0&0&4\\
0&1&0&2&4&6&7\\
0&1&1&0&2&2&4\\
0&1&1&1&2&4&0\\
0&1&1&2&0&3&4\\
0&1&2&0&2&4&0\\
0&1&2&1&4&5&7\\
0&1&2&2&2&4&1\\
0&1&2&2&3&4&4\\
0&1&2&4&1&6&1\\
0&1&2&4&4&3&5\\
0&1&2&4&4&7&3\\
1&0&0&2&1&4&6\\
1&0&0&2&3&4&2\\
1&0&0&2&4&2&7\\
1&0&1&2&4&1&3\\
1&0&2&0&4&6&7\\
1&0&2&1&0&3&4\\
1&0&2&1&2&4&7\\
1&0&2&4&3&5&5\\
1&0&2&4&4&4&6\\
1&1&0&2&2&0&4\\
1&1&0&2&4&5&2\\
1&1&2&1&4&3&0\\
1&1&2&2&4&0&5\\
1&1&2&4&2&1&0\\
1&1&2&4&2&3&1\\
1&2&0&0&3&1&4\\
1&2&0&2&2&4&6\\
1&2&0&2&4&2&2\\
1&2&0&3&4&1&3\\
1&2&1&0&3&2&4\\
1&2&1&0&4&7&5\\
1&2&1&1&4&1&5\\
1&2&1&2&1&4&2\\
1&2&1&4&0&2&6\\
1&2&1&4&5&0&7\\
1&2&2&0&4&3&3\\
1&2&2&4&2&1&7\\
1&2&2&4&6&0&5\\
1&2&3&1&4&6&5\\
1&2&3&4&2&4&1\\
1&2&3&4&4&1&3\\
1&2&3&4&6&1&0\\
1&2&4&0&2&6&6\\
1&2&4&0&4&1&6\\
1&2&4&1&1&5&7\\
1&2&4&1&2&6&5\\
1&2&4&1&5&3&3\\
1&2&4&2&4&3&0\\
1&2&4&2&6&7&2\\
1&2&4&3&5&4&4\\
1&2&4&3&7&2&7\\
1&2&4&4&0&6&7\\
1&2&4&4&3&4&3\\
1&2&4&4&4&0&1\\
1&2&4&4&6&0&6\\
1&2&4&5&1&2&2\\
1&2&4&5&3&1&1\\
1&2&4&5&6&6&3\\
1&2&4&6&0&5&0\\
1&2&4&6&0&7&3\\
1&2&4&6&1&3&4\\
1&2&4&6&2&4&0\\
1&2&4&6&5&1&7\\
1&2&4&7&0&0&7\\
1&2&4&7&4&0&4\\
1&2&4&7&7&5&4\\
\end{supertabular}
}\\

\end{document}